\documentclass{gtpart}     
\agtart
\usepackage{pinlabel}  
\usepackage{amsfonts}
\usepackage{graphicx}
\usepackage{amscd}
\usepackage{graphicx}
\usepackage{microtype}
\usepackage{enumerate}
\usepackage{subfig}

\title{Bounded orbits and global fixed points
for groups acting on the plane}

\author{}
\givenname{Kathryn}
\surname{Mann}
\address{}
\email{mann@math.uchicago.edu}
\urladdr{}


\arxivreference{}
\arxivpassword{}

\volumenumber{}
\issuenumber{}
\publicationyear{}
\papernumber{}
\startpage{}
\endpage{}
\doi{}
\MR{}
\Zbl{}
\received{}
\revised{}
\accepted{}
\published{}
\publishedonline{}
\proposed{}
\seconded{}
\corresponding{}
\editor{}
\version{}

\newtheorem{theorem}{Theorem}[section]    
\newtheorem{lemma}[theorem]{Lemma}          
\theoremstyle{definition}
\newtheorem{corollary}[theorem]{Corollary}

\newtheorem{example}[theorem]{Example}
\newtheorem{fact}[theorem]{Fact}
\newtheorem*{remark}{Remark}             
\newtheorem*{LOcor}{Corollary \ref{LO}}

\renewcommand{\R}{\mathbb{R}}

\renewcommand{\Z}{\mathbb{Z}}
\newcommand{\N}{\mathbb{N}}
\renewcommand{\Q}{\mathbb{Q}}
\newcommand{\del}{\partial}

\newcommand{\pend}{\mathcal{P}}

\newcommand*{\defeq}{\mathrel{\vcenter{\baselineskip0.5ex \lineskiplimit0pt
                     \hbox{\scriptsize.}\hbox{\scriptsize.}}}%
                     =}

\DeclareMathOperator{\SL}{SL}

\DeclareMathOperator{\Homeo}{Homeo}

\DeclareMathOperator{\GL}{GL}
\DeclareMathOperator{\interior}{int}
\DeclareMathOperator{\Aff}{Aff}

\DeclareMathOperator{\princ}{pr}

\begin{document}

\begin{abstract}    

Let $G$ be a group acting on $\R^2$ by orientation-preserving homeomorphisms.  We show that a tight bound on orbits implies a global fixed point.  Precisely, if for some $k>0$ there is a ball of radius $r > \frac{1}{\sqrt{3}}k$ such that each point $x$ in the ball satisfies $\| g(x) - h(x)\| \leq k$ for all $g, h \in G$, and the action of $G$ satisfies a nonwandering hypothesis, then the action has a global fixed point.   In particular any group of measure-preserving, orientation-preserving homeomorphisms of $\R^2$ with uniformly bounded orbits has a global fixed point.  The constant $\frac{1}{\sqrt{3}}k$ is sharp.

As an application, we also show that a group acting on $\R^2$ by diffeomorphisms with orbits bounded as above is left orderable. 

\end{abstract}

\maketitle


\setcounter{section}{0}

\section{Introduction} 

The Brouwer plane translation theorem states that any orientation-preserving homeomorphism of the plane with a bounded orbit must have a fixed point.  We would like to extend this kind of fixed-point theorem to groups other than $\Z$ acting on the plane, finding some minimal restrictions on a group of homeomorphisms of $\R^2$ which guarantee a \emph{global fixed point} -- a point $x \in \R^2$ such that $gx =x$ for all $g \in G$.  

Several papers address the question of global fixed points for abelian groups.  Of note is recent work of Franks, Handel, and Parwani \cite{FHP}. They prove, using very different methods from those we employ in this paper, that any finitely generated abelian group of orientation-preserving $C^1$ diffeomorphisms of $\R^2$ leaving invariant a compact set has a global fixed point.   

Franks and others have also studied the fixed point sets of $\Z$-actions by measure-preserving homeomorphisms.  Much of this work focuses on measure-preserving homeomorphisms of surfaces, but since the universal cover of any closed, oriented surface of genus $g \geq 1$ is the plane, it is intimately related with the study of fixed points of homeomorphisms of the plane.   Theorem \ref{thm} 
here could be viewed as a first step towards generalizing these results to groups other than $\Z$.  

Our main theorem concerns the existence of global fixed points of an arbitrary group $G$ acting on the plane by orientation-preserving homeomorphisms, without making assumptions on the algebraic structure of $G$.  Of course, some restriction on the group action is necessary to ensure a global fixed point, for we must exclude group elements acting (for example) by translations.  To rule out such examples, we require that the group have \emph{uniformly bounded orbits} on a bounded set of sufficient size.   We also require a nonwandering hypothesis to eliminate more pathological counterexamples.  What is surprising is that these two conditions suffice:  

\begin{theorem} \label{thm}
Let $G$ be any group acting on $\R^2$ by orientation-preserving homeomorphisms.  Suppose that there exists a constant $k$ and a ball $B \subset \R^2$ of radius $r > \frac{1}{\sqrt{3}}k$ such that for all $g, h \in G$ and $x \in B$ we have $\|gx -hx\| \leq k$.  Suppose additionally that no point in the wandering set for any $g \in G$ is contained in $B$ or any bounded component of $\R^2 \setminus GB$.
Then $G$ has a global fixed point.  
\end{theorem}

The \emph{wandering set} for a homeomorphism $g$ is the set of points $x \in \R^2$ such that there is some open set $V$ containing $x$ with all translates $g^n(V)$ (for $n \in \N$) disjoint from $V$.    
The condition that no point in the wandering set for $g$ has forward orbit contained in $B$ or a bounded component of $\R^2 \setminus GB$ is true in particular if $G$ is a group of measure preserving transformations (for any locally finite measure with full support), or if $G$ has \emph{uniformly bounded contraction} on $B$, meaning that there is some $\epsilon > 0$ so that $\|gx-gy\| > \epsilon\|x-y\|$ for any $x, y \in B$ and $g \in G$.  

The constant $\frac{1}{\sqrt{3}}k$ in the statement of Theorem \ref{thm} is sharp.  To show this, in section \ref{sharpness} we construct a finitely generated group acting on $\R^2$ by measure-preserving homeomorphisms (hence satisfying the hypothesis on the wandering set) with every orbit bounded by $k$ on a ball of radius $r = \frac{1}{\sqrt{3}}k$, but no global fixed point.  
We also discuss in more detail the hypothesis on the wandering set and give a counterexample when this condition is not satisfied.   

The paper concludes with an application of our proof technique to left invariant orders.  We prove
\begin{corollary} \label{LO}
Let $G$ be a group acting on $\R^2$ by $C^1$ diffeomorphisms, and satisfying the property that there is some constant $k$ and a ball $B \subset \R^2$ of radius $r > \frac{1}{\sqrt{3}}k$ such that $\| g(x) - h(x)\| \leq k$ for any $g, h \in G$ and $x \in B$.  Then $G$ is left orderable.
\end{corollary}
The definition of left orderable is given in Section \ref{section LO} along with some discussion of left orderability.  The constant $\frac{1}{\sqrt{3}}k$ is again sharp.

\section{Proof of Theorem \ref{thm}}

The use of prime ends in the second part of this proof is inspired by Calegari's work on circular orders on subgroups of $\Homeo^+(\R^2)$ in \cite{Ca}.

Let $B$ be an open ball in $\R^2$ of radius $r > \frac{1}{\sqrt{3}}k$ satisfying the hypotheses of the theorem.  Assume without loss of generality that $B$ is centered at the origin.  Its orbit $\displaystyle GB := \bigcup \limits_{g\in G} gB$ is a $G$-invariant set.  We will study the action of $G$ on the boundary of $GB$ and find a fixed point there.  First, note that our hypothesis that $\|gx - hx\| < k$ for each $x \in B$ and $g, h \in G$ implies that $GB$ is connected.  In fact, the following stronger statement is true: 

\begin{lemma}\label{lemma}
Let $G$ be any group acting on $\R^2$ by orientation-preserving homeomorphisms.  Suppose that there is a closed ball $B$ of radius $r$ such that for all $x \in B$, the orbit $Gx$ has diameter strictly less than $2r$.  Then $GB$ is a connected set, in fact it is path-connected.  
\end{lemma}

\begin{proof} We will show that $gB \cap B \neq \emptyset$ for any $g \in G$.  To see this, assume for contradiction that $gB \cap B = \emptyset$ and consider the map $\phi: S^1 \to S^1$ given by $\phi(x) = \frac{g(rx)}{\|g(rx)\|}$, where $S^1$ is the unit circle centered at the origin in $\R^2$ .  This map is well defined since $0 \in B$ so $g(rx) \neq 0$ for any point $rx \in \del B$.  Since the orbit of each point in $\del B$ has diameter strictly less than $2r$ (in fact, less than $\sqrt{3}r$, but $2r$ is all we need) and $gB \cap B = \emptyset$, it follows that $\phi(x) \neq -x$ for all $x \in S^1$.  This means that $\phi$ has degree $1$; indeed, 
$$\Phi(x,t) = \frac{t \phi(x) + (1-t)x}{\|t \phi(x) + (1-t)x\|}$$ 
is a homotopy between $\phi$ and the identity map.  However, our assumption that $gB \cap B = \emptyset$ implies that $\phi$ must have degree zero.  
\end{proof}

Thus, $GB$ is a connected, $G$-invariant set containing $B$ and with compact closure.  Moreover, since each point in $GB$ is in the orbit of some point of $B$, we have $\|gx - hx\| \leq k$ for all $x \in GB$ and $g, h \in G$.  The same inequality holds for each point $x$ in the boundary $\del(GB)$, since for each $g$ and $h \in G$, the set $\{x \in \R^2 : \|gx - hx\| \leq k \}$ is closed.  We may further assume that $GB$ is simply connected, for its complement contains only one unbounded component, and so the union of $GB$ with all bounded components of its complement is a connected, simply-connected, $G$-invariant set containing $B$.  Since the boundary of this simply connected region is a subset of the original boundary, each point $x$ in the new boundary also satisfies $\|gx - hx\| \leq k$ for all $g, h \in G$.  We now examine the action of $G$ on $\del(GB)$.  

Assume first that $\del(GB)$ is homeomorphic to a circle.  Since $G$ preserves $GB$, it preserves $\del(GB)$, and acts on $\del(GB)$ by orientation-preserving homeomorphisms of this circle.  We will find a fixed point on $\del(GB)$ for this action. 
Let $x$ be any point in $\del(GB)$.  Since the orbit of $x$ lies in the complement of $B$, a ball of radius $r > \frac{k}{\sqrt{3}}$, and the maximal distance between any two points in the orbit is $k$, the orbit lies entirely in a sector of angle $\frac{2\pi}{3}$ from the origin.  By this, we mean a set of the form $\{s e^{i\theta} : \alpha < \theta < \alpha+\frac{2\pi}{3}, s>0 \}$ for some real number $\alpha$.  Without loss of generality, we may assume the orbit lies in the sector $S \defeq \{ se^{i\theta} : \frac{- \pi}{3} < \theta < \frac{\pi}{3}, s>0 \}$. 

The complement of the closure of the orbit in the circle, $\del(GB) \setminus \overline{Gx}$, consists of a union of disjoint open intervals which are permuted by $G$. Say that an interval $I$ is \emph{spanning} if it is possible to join the endpoints of $I$ with a path in $S \setminus B$ so that the resulting loop is homotopic to $\del B$ in $\R^2 \setminus B$.  See Figure \ref{spanning}
for an illustration.  More generally, we say that any set $X \subset (\R^2 \setminus B)$ is spanning if there is some path $\gamma$ in $S \setminus B$ such that $\gamma \cup X$ is homotopic to $\del B$ in $\R^2 \setminus B$.  

  \begin{figure*}
         \labellist 
  \small\hair 2pt
   \pinlabel $I$ at 40 200 
   \pinlabel $GB$ at 100 210
   \pinlabel $B$ at 115 165
   \pinlabel $S$ at 230 175 
   \pinlabel $x$ at 215 126
   \endlabellist
  \centering
    \subfloat[$I \subset (\del(GB) \setminus \bar{Gx})$ is a spanning interval. \, \, \, \, \, \, \, \, \, \, Dots indicate points in the orbit of $x$]{ \label{spanning} 
    \includegraphics[width=2.2in]{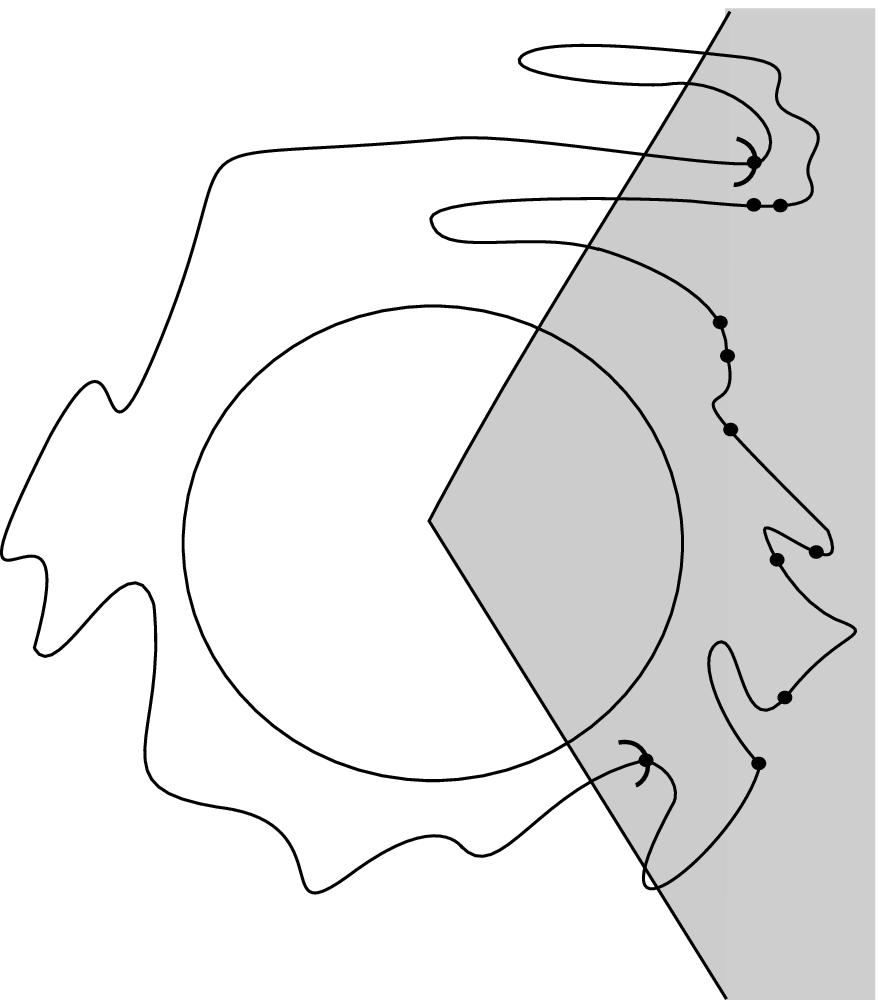}}
             \labellist 
  \small\hair 2pt
   \pinlabel $A$ at 40 240 
   \pinlabel $I$ at 62 210
   \pinlabel $x_0$ at 235 265
   \pinlabel $x_1$ at 220 35 
   \endlabellist
    \subfloat[If there are multiple spanning intervals, we have a spanning arc $A \subset int(GB)$ ``exterior" to some spanning interval $I \subset \del(GB)$] 
   { \label{unique} \includegraphics[width=2.25in]{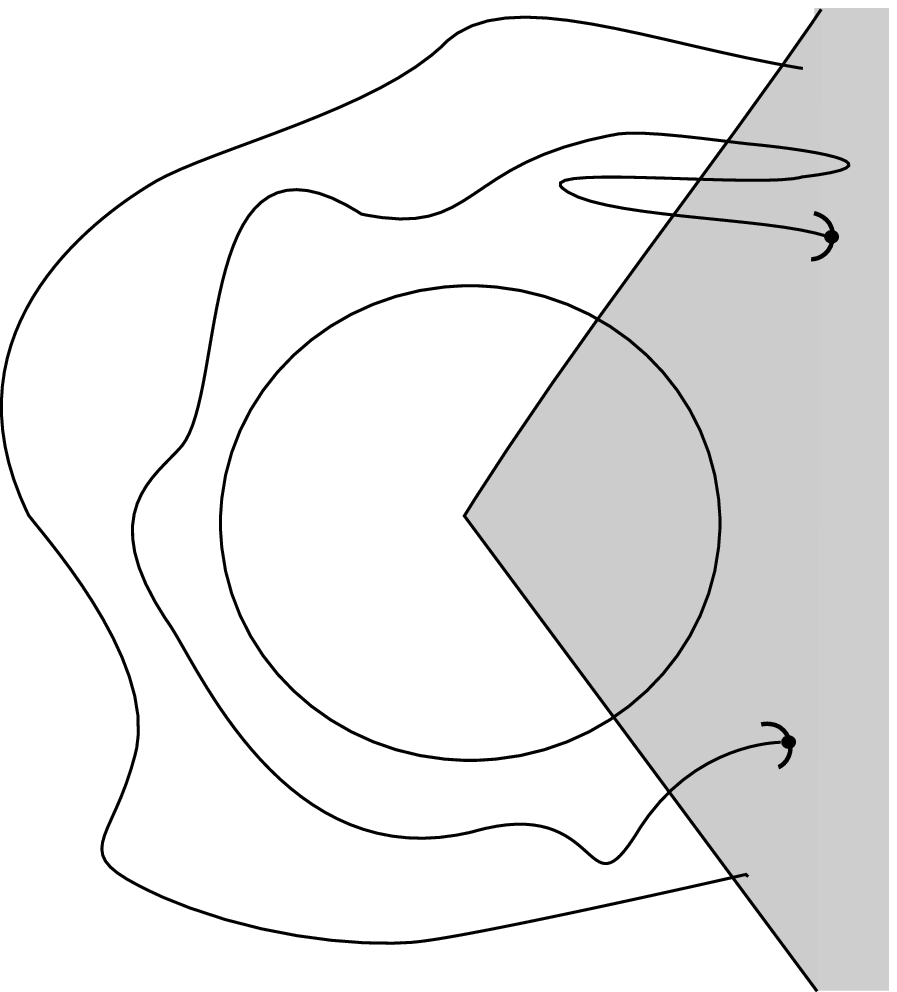}}
    \caption{Spanning intervals}
  \end{figure*}

We claim that there is only one spanning interval in $\del(GB) \setminus \overline{Gx}$.  
Since $GB$ contains $B$, there must be at least one.  If there is more than one, then let $I$ denote a spanning interval such that the unbounded component of $\R^2 \setminus (S \cup I)$ contains another spanning interval.  In particular, this implies that there is an arc $A$ of points \emph{in the interior of} $GB$ contained in the unbounded component of $\R^2 \setminus (S \cup I)$ that is also a spanning arc.  Let $x_0 = A \cap \{ a e^{i \frac{- \pi}{3}}: a \in \R \}$ and $x_1 =   A \cap \{ a e^{i \frac{\pi}{3}} : a \in \R\}$ be the endpoints of this arc.  See Figure \ref{unique}.
For some elements $g_0, g_1 \in G$, we have $x_i \in g_iB$.  By the proof of Lemma \ref{lemma}, $B \cup g_i B$ is path-connected, so there is a path $\gamma_i$ from $x_i$ to the origin in $B \cup g_iB$.  Since $GB$ is simply connected, one of these paths $\gamma_i$ must have $\gamma_i \cap (\R^2 \setminus S)$ as a spanning set: otherwise $A \cup \gamma_0 \cup \gamma_1$ is the boundary of a disc containing $I$, so is not contractible in $GB$. 

Thus, for either $i=0$ or $i=1$, we have a spanning arc $\gamma_i \cap (\R^2 \setminus S)$ in the unbounded component of $\R^2 \setminus (S \cup I)$ contained entirely in $g_i B \cup B$.  Considering now the boundary of $g_i B$, this implies that $\del (g_i B) \cap (\R^2 \setminus (S \cup B))$ contains \emph{two} disjoint spanning arcs (Figure \ref{giB}).  
However, the restriction $\|g_i x - x\| \leq k$ for all $x \in \del B$ implies that $\del (g_i B)$ has only a \emph{single} spanning arc: an easy calculation using the restriction shows that any subarc of $g_iB$ contained in $\R^2 \setminus (S \cup B)$ that spans must contain the image of the point $(-r, 0) \in \del B$ under the homeomorphism $g_i$.  This gives the desired contradiction, proving our claim that there is only one spanning interval in $\del(GB) \setminus \overline{Gx}$.

  \begin{figure*}
   \labellist 
  \small\hair 2pt
   \pinlabel $I$ at 50 140 
   \pinlabel $g_iB$ at 240 158
   \pinlabel $B$ at 122 122 
   \endlabellist
  \centerline{
    \mbox{\includegraphics[width=2.2in]{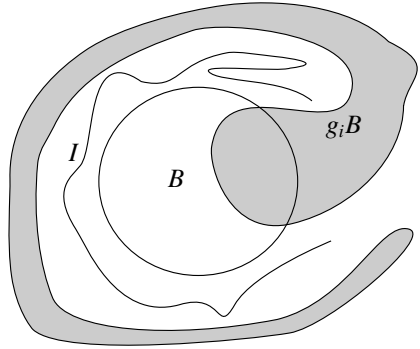}}}
 \caption{The boundary of $g_iB$ has two spanning arcs, violating $\|g_i x - x\| \leq k$ on $\del GB$}
  \label{giB}
  \end{figure*}

Abusing notation slightly, let $I$ denote the spanning interval of $\del(GB) \setminus \overline{Gx}$.  We now employ the orbit bound condition once more to show that each element of $G$ fixes $I$.  Assume for contradiction that for some $g$ we have $gI \cap I = \emptyset$.  By uniqueness of spanning intervals, $gI$ is not spanning.  Using the fact that $I$ is spanning, connect the endpoints of $I$ with a path in $S \setminus B$ to form a loop $L$ homotopic in $\R^2 \setminus B$ to the boundary of $B$.  Since $gI$ is not spanning, we may then connect its endpoints by a path in $S \setminus B$ to form a loop $L'$ nullhomotopic in $\R^2 \setminus B$.  We can also easily choose these paths to avoid the origin.  Extend the homeomorphism $g: I \to gI$ arbitrarily on these paths, giving a homeomorphism $f: L  \to L'$.  By construction, no point is mapped under $f$ to an antipodal point -- a point $a e^{i \theta} \in L \subset \R^2$ cannot have $f(a) = b e^{-i \theta}$ for any real number $b$.  This is true because of the orbit bound on points in $I$ mapping under $g$, and holds on $L \cap S$ since $f(L\cap S)$ is also contained in $S$.  See figure \ref{loops}.

  \begin{figure*}[h!]
   \labellist 
  \small\hair 2pt
   \pinlabel $I$ at 10 210 
   \pinlabel $gI$ at 77 254
   \pinlabel $0$ at 95 163 
   \endlabellist
  \centerline{ \mbox{
 \includegraphics[scale=.5]{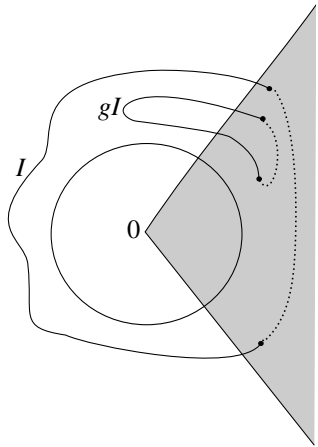}}}
 \caption{Intervals $I$ and $gI$ closed into loops}
  \label{loops}
  \end{figure*} 

Now a degree argument similar to the proof of Lemma \ref{lemma} shows that $f(L) =L'$ cannot be nullhomotopic in $\R^2 \setminus B$, giving a contradiction.  In detail, if $\pi : \R^2 \setminus \{0\} \to \del B$ is radial projection $ a e^{i \theta} \mapsto r e^{i \theta}$ then
$$\Phi(x,t) = \frac{r \left(t f(x) + (1-t)(x)\right)}{\|t f(x) + (1-t)(x)\|}$$ 
is a homotopy between the radial projection map $\pi : \R^2 \setminus {0} \to \del B$ defined by $\pi: a e^{i \theta} \mapsto r e^{i \theta}$, and the map $\pi \circ f: L \to \del B$.  (Here, we use that $L$ and $L'$ avoid the origin, and that no point maps under $f$ to an antipodal point).   However, no such homotopy exists if the loop $L'$ is nullhomotopic in $\R^2 \setminus B$.  It follows that $I$ is a $G$-invariant interval, and the endpoints of $I$ (or endpoint, if $I = \del(GB) \setminus \{x\}$) are global fixed points for $G$.  This proves the theorem when $\del(GB)$ is homeomorphic to a circle.

In the case where $\del(GB)$ is not homeomorphic to a circle, we can modify the argument above using the theory of prime ends.  Prime ends are a tool from conformal analysis developed to study the boundary behavior of conformal maps defined on open regions in the complex plane.  Here we will use them to turn the action of $G$ on $\del(GB)$ into an action on the circle.  There are several equivalent definitions of prime ends; we follow Pommerenke.  The reader may refer to \cite{Po} or \cite{Ma} for more details.  

For an open set $U$ in $\R^2$, define a \emph{crosscut} in $U$ to be an open Jordan arc $C$ such that $\overline{C}\setminus C$ consists of one or two points on the boundary of $U$.  Define a \emph{null chain} in $U$ to be a sequence of crosscuts $C_i$ satisfying 
\begin{enumerate}
\item $\overline{C_i}\cap \overline{C_{i+1}} = \emptyset$
\item $C_i$ separates $C_0$ from $C_{i+1}$ in $U$
\item The diameter of $C_i$ converges to 0 as $i$ goes to infinity.  
\end{enumerate} 
A \emph{prime end} of $U$ is an equivalence class of null chains, where $\{C_i\}$ is equivalent to $\{C'_i\}$ if for all $m$ sufficiently large there is a $n$ such that $C_m$ separates $C'_0$ from $C'_n$, and $C'_m$ separates $C_0$ from $C_n$.   We let $\pend(U)$ denote the set of all prime ends of $U$.  We need only concern ourselves with the case where $U$ is simply connected and has simply connected, compact closure and will also assume that the boundary of $U$ is equal to the \emph{frontier} $\overline{U} \setminus U$, although in general one needs to work with the frontier rather than the usual boundary.  

If $\del U$ is locally connected, the prime ends of $U$ correspond exactly to proper homotopy classes (fixing the boundary) of rays in $U$.  In particular, for the unit disc $D$ we have $\pend(D) = S^1$.  When $\del U$ is not locally connected prime ends are not so well behaved.  To each prime end $e \in \pend(U)$ we can associate a \emph{subset} of $\del U$ called the \emph{principal points} of $e$.  A \emph{principal point} for $e$ is an accumulation point of a sequence $x_i$ where $x_i \in C_i$ for some chain $\{C_i\}$ representing $e$.  It is not hard to show that the set of principal points of any prime end is nonempty and connected.  It is also known that the set of points in $\del U$ which are the \emph{sole} principal point of some prime end is a dense set in $\del U$.  But it is possible for a single prime end to have a continuum of principal points, and for infinitely many distinct prime ends to share the same set of principal points.  So principal points do not define a one-to-one correspondence of prime ends with points on the boundary.   

However, there is a noncanonical one-to-one correspondence between prime ends of $U$ and points of $S^1$.  This correspondence is obtained by choosing a conformal map $f: U \to D$, which exists by the Riemann mapping theorem.  The map $f$ induces a map $\hat{f}$ on prime ends, since for any null chain $\{C_i\}$ representing a prime end of $U$ the sequence $\{f(C_i)\}$ is a null chain in $D$, so represents a prime end of $D$.  It is a theorem of Carath\'eodory that this map is a bijection $\pend(U) \to S^1$ \cite{Car}.   Moreover, there is a natural topology on $\pend(U)$ that makes $\hat{f}$ a homeomorphism and an extension of the homeomorphism $f: U \to D$.  A basis for this topology consists of sets of the form $\hat{V} := \{ [\{C_i\}] : C_n \subset V \text{ for } n \text{ sufficiently large} \}$, where $V$ is an open set in $U$.  In particular, when $U = D$ this topology on $\pend(D)$ agrees with the usual topology on $S^1$.  

Now to apply this theory to our situation.  Let $U = GB$, which we may again assume to be simply connected.  Replacing $GB$ with $\interior(\overline{GB})$ if necessary, we can also assume that $\del U = \overline{U} \setminus U$; for $U$ and $\del U$ will still be $G$-invariant sets.  Fix a conformal map $f: U \to D$.  Since $G$ acts by orientation-preserving homeomorphisms on $U$, the induced action $\hat{f}G \hat{f}^{-1}$ on $S^1$ is by orientation-preserving homeomorphisms as well.  
As in the previous part of the proof, we will find an invariant interval and therefore a fixed point for the action on $S^1$.  The idea is to use the condition $\|gx - hx\| \leq k$ on principal points of prime ends in $\R^2$ to reproduce the argument from the simple case above.  

Let $e$ be a prime end of $U$ with a single principal point $x \in \del U$.  As in the case where $\del(GB)$ is homeomorphic to a circle, we may assume that the orbit of $x$ lies in the sector $S = \{ s e^{i\theta} : \frac{- \pi}{3} < \theta < \frac{\pi}{3} \}$. 

The complement of the closure of the orbit of $e$ in the circle of prime ends consists of a union of open intervals in which are permuted by $G$.  We modify our original definition of ``spanning" to the following:  For any interval $I$ of prime ends, let $ \princ(I) \subset \R^2$ denote the union of all principal points of prime ends in $I$.  Then $I$ is said to be \emph{spanning} if there is some path $\gamma$ in $S \setminus B$ such that $ \princ(I) \cup \gamma$ is homotopic to $\del B$ in $\R^2 \setminus B$.  The argument that there is a unique spanning interval and that this interval must be fixed by the action of $G$ can now be carried through just as in the case where $\del(GB)$ is homeomorphic to a circle.  One simply works with the sets $ \princ(I)$ in the plane instead of the intervals themselves.  In the final step, one needs to replace $ \princ(I)$ with (for instance) a projection onto a nearby Jordan curve (and $ \princ(gI)$ as well) in order to close these into loops $L$ and $L'$.  But this is easily done, and the rest of the proof applies nearly verbatim.  

Thus there is some $G$-invariant interval for the action $\hat{f}G \hat{f}^{-1}$ on $\pend(U) = S^1$.   Since $\hat{f}G \hat{f}^{-1}$ acts by orientation-preserving homeomorphisms on $S^1$, the endpoint(s) of this interval in $\pend(U)$ must be fixed by $G$.  However, a fixed prime end does not ensure a fixed point for the action of $G$ on $\R^2$: a counterexample is given in Example \ref{ex2} below.  In order to find a fixed point for the action of $G$ on the plane we must use the nonwandering hypothesis.  

Let $e$ be a prime end fixed by $G$ and let $\{C_i\}$ be a null chain representing $e$, and let $p$ be a principal point of $e$.  Then for any $g$ in $G$, $\{g(C_i)\}$ defines an equivalent null chain.  Following an argument in \cite{FL}, we will show that $g(C_i) \cap C_i \neq \emptyset$.  This together with the fact that the diameters of the $C_i$ and $g(C_i)$ tend to zero shows that any limit point of a sequence $x_i \in C_i$, and in particular $p$, is a point fixed by $g$.  To see that $g(C_i) \cap C_i \neq \emptyset$, note that each crosscut $C_i$ divides $U$ into two components, and let $U_i$ be the component that contains $C_n$ for $n>i$.  If $C_i \cap g(C_i) = \emptyset$, the equivalence of the null chains $\{(C_i)\}$ and $\{g(C_i)\}$ implies that $g(U_i) \subset U_i$ or $U_i \subset g(U_i)$.  In the first case, there is an open set $V \subset U_i \setminus gU_i$, and so the images $g^n(V)$ are all disjoint from $V$ and contained in $U_i$, contradicting our hypothesis that no wandering point for $g$ had forward orbit in $U$.  If instead $U_i \subset g(U_i)$, the sets $g^{-n}(V)$ are disjoint from $V$, which is again a contradiction.  Since $g$ was arbitrary, $p$ is a global fixed point for the action of $G$ on $\R^2$.  

\hfill $\square$

\section{Sharpness of the hypotheses} \label{sharpness}

In this section we show that the constant $\frac{1}{\sqrt{3}}k$ in the statement of Theorem \ref{thm} is sharp and that a hypothesis on the wandering set is necessary.    

To demonstrate sharpness of the constant, here is an example of an action on $\R^2$ by a finitely generated group of measure-preserving, orientation-preserving homeomorphisms with all orbits on a ball of diameter $\frac{1}{\sqrt{3}}k$ bounded by $k$ but no global fixed point.  Since the action preserves the standard Lebesgue measure on $\R^2$, it satisfies the nonwandering hypothesis in the theorem.  
\begin{example} \label{ex1}
Let $G \subset \Homeo^+(\R^2)$ be generated by $a$ and $b$, where $a$ is a rotation by $2\pi/3$ about the origin, and $b$ is any orientation-preserving, measure-preserving homeomorphism that does not fix the origin, but fixes each point in the complement of some small disc of radius $\rho$ centered at the origin.  Since the only point fixed by $a$ is 0, but 0 is not fixed by $b$, there is no global fixed point.  However, every point in $\R^2$ has a bounded $G$-orbit, for the disc of radius $\rho$ is $G$-invariant and each point in the complement of the disc has a three point orbit.  Moreover, on any ball $B$ centered at $0$ and of radius $\frac{1}{\sqrt{3}}k$ with $\frac{1}{\sqrt{3}}k > \rho$, each point $x$ in $B$ satisfies $\|gx - hx\| \leq k$ for all $g, h \in G$.  

\end{example}

It remains to discuss the hypothesis on the wandering set of $G$.   Note first that the proof of Theorem \ref{thm} still goes through if rather than assuming that no point in the wandering set for $g$ is contained in $B$ or any bounded component of $\R^2 \setminus GB$, we require only that for each $g \in G$, there is a neighborhood of $\del U$ in $U$ that does not contain the forward orbit of any wandering point for $g$.  In the case where $G$ is finitely generated, we need only require this of the generators.  

This assumption is necessary.  There are examples of homeomorphisms of $\R^2$ that leave invariant a bounded domain $U$ and fix a prime end of $U$ but do not fix any point on the boundary of $U$.  One such example is given by Barge--Gillette in \cite{BG}.  Their key construction is a region which spirals infinitely many times around a circle.  A homeomorphism that rotates the circle and shifts points along the spiraling arm will fix the prime end corresponding to the spiraling arm, but not fix any point on the boundary of the arm (Figure \ref{arm}).  

   \begin{figure*}[h]
  \centerline{
    \mbox{\includegraphics[width=1.0in]{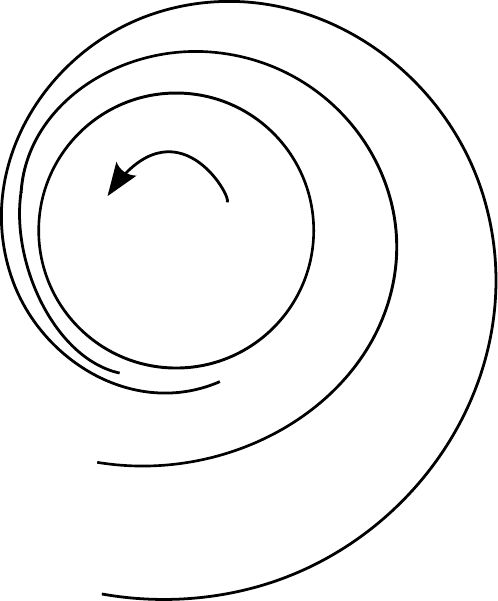}}}
\caption{A homeomorphism of a region spiraling infinitely many times around a circle}
  \label{arm}
  \end{figure*}
  
  \newpage

Using this as a building block, we will construct a group of orientation-preserving homeomorphisms of $\R^2$ generated by two elements such that for all points $x$ in a ball of radius $r > \frac{1}{\sqrt{3}}k$ we have $\|gx - hx\| \leq k$, but violating the nonwandering hypothesis and with no global fixed point.    

   \begin{figure*}[h]
  \centerline{
    \mbox{\includegraphics[width=3.2in]{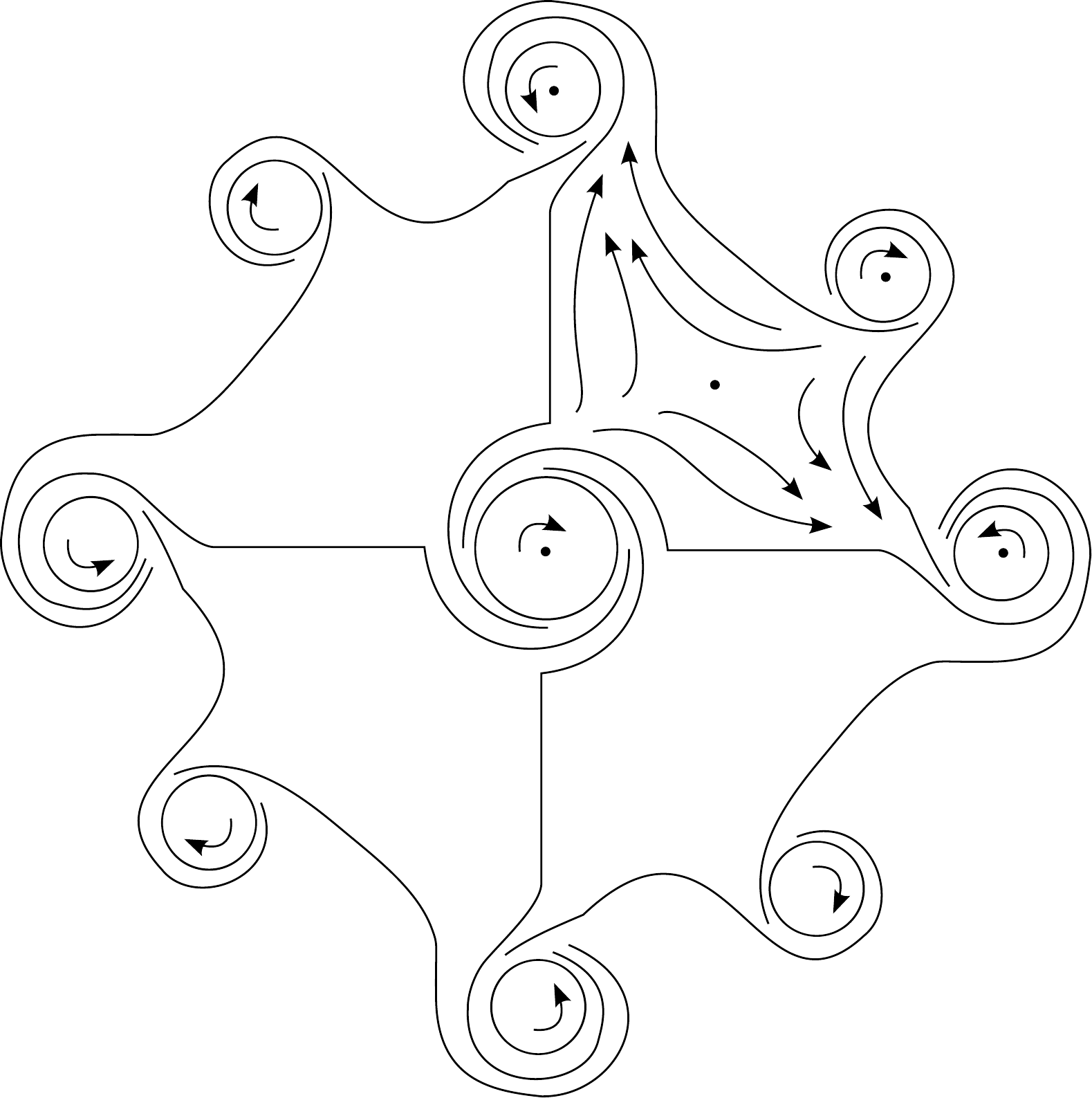}}}
 \caption{The action of $a$ on a region in $\R^2$}
  \label{ctx}
  \end{figure*}

\begin{example} \label{ex2} Consider a group $G$ generated by two homeomorphisms, $a$ and $b$.  The homeomorphism $a$ acts on the region shown in Figure \ref{ctx} by rotating each circle a fixed amount in the direction indicated, and shifting other points along the flow lines, which spiral infinitely many times around each circle.  The action of $a$ on the complement of the region can be arbitrary.  Now let $b$ be any homeomorphism which does not fix any of the (isolated) fixed points of $a$ and acts as the identity outside of a small neighborhood of each fixed point of $a$.  The region shown in Figure \ref{ctx} is then a $G$-invariant set.  Let $k$ be such that the diameter of the largest ball contained in the region is $\frac{1}{\sqrt{3}}k + \epsilon$ for some $\epsilon > 0$.  By making the circles rotated by $a$ sufficiently small, we can ensure that the diameter of the orbit of any point in the ball is bounded by $k$.  However, there is no global fixed point for the action.  
\end{example}

\section{Actions by diffeomorphisms and left-invariant orders} \label{section LO}

As an application of the proof of Theorem \ref{thm}, we show in this section that a group acting on the plane by orientation-preserving diffeomorphisms with orbits bounded as in the statement of Theorem \ref{thm} is left orderable.  The nonwandering hypothesis is not necessary here.  This result (Corollary \ref{LO}, restated below) is closely related to, but does not follow immediately from, the results of Calegari in \cite{Ca} where it is shown that any group acting on the plane by $C^1$ diffeomorphisms and leaving invariant two disjoint, compact, simply connected sets is left orderable, and that any group acting on $\R^2$ with a bounded orbit is circularly orderable (a weaker notion than left orderability).  The constant $\frac{1}{\sqrt{3}}k$ is again sharp.  

A group $G$ is said to be \emph{left orderable} if it admits a total order $<$ such that for any $a, b \in G$, $a<b$ if and only if $ca <cb$ for all $c \in G$.  It is not difficult to show that for countable groups, being left orderable is equivalent to admitting a faithful representation into $\Homeo^+(\R)$.  Some examples of left orderable groups are $\Z$, surface groups, free groups, the group $\Aff^+(\R)$ of orientation-preserving affine transformations of the line, and braid groups (see \cite{Ca}).  In fact, any finitely generated, residually torsion-free, nilpotent group is left orderable.  Additionally, a subgroup of a left orderable group is left orderable, for it inherits any left order on the group.  Some groups which are not left orderable include groups with torsion elements, finite index subgroups of $\SL(n,\Z)$ for $n \geq 3$, and more generally, arithmetic subgroups of simple algebraic groups over $\Q$ of $\Q$-rank $\geq 3$ \cite{Wi}.  

We will use the proof of Theorem \ref{thm} to show that a group acting on $\R^2$ with orbits bounded as in Theorem \ref{thm} is left orderable.  The tools that we need for this are the following basic fact and two theorems:

\begin{fact}[Lemma 2.19 in \cite{Ca}] \label{fact}
Left orderability behaves well under short exact sequences: if $A$ and $C$ are left orderable groups and we have a short exact sequence 
$$1 \to A \to B \to C \to 1$$ 
then $B$ must also be left orderable. 
\end{fact}

\begin{theorem}[Burns--Hale \cite{BH}] If a group $G$ satisfies the condition that every finitely generated subgroup of $G$ surjects to an infinite left orderable group, then $G$ is left orderable.  
\end{theorem}

\begin{theorem}[Thurston stability \cite{Th}]
Let $G$ be a group of $C^1$ diffeomorphisms of $\R^n$ with global fixed point $p \in \R^n$.  Let $D : G \to \GL(n,\R)$ be the homomorphism given by taking the derivative of the action of each element $g$ at $p$, i.e. considering the linear action of $g$ on $T_p\R^n$. Then any nontrivial finitely generated subgroup of $\ker(D)$ surjects to $\Z$. 
\end{theorem}

\begin{remark} The condition that every finitely generated subgroup of a group $G$ surjects to $\Z$ is called \emph{local indicability}.  It follows immediately from Burns--Hale is that locally indicable implies left orderable; however, the converse is not true.  
\end{remark}

Combining Burns--Hale and Thurston stability shows in particular that $\ker(D)$ is left orderable.  We can now prove our orderability result, which we restate here for convenience.  

\begin{LOcor}
Let $G$ be a group acting on $\R^2$ by $C^1$ diffeomorphisms, and satisfying the property that there is some constant $k$ and a ball $B \subset \R^2$ of radius $r > \frac{1}{\sqrt{3}}k$ such that for any $g, h \in G$ and $x \in B$ we have $\| g(x) - h(x)\| \leq k$.  Then $G$ is left orderable.
\end{LOcor}

\begin{proof}
The proof of Theorem \ref{thm} shows that $G$ leaves invariant a simply connected set $U$ and that there is a global fixed point for the action of $G$ on the set of prime ends of $\del U$, which is homeomorphic to a circle.  Let $H$ be any finitely generated subgroup of $G$.  If the action of $H$ on $\pend(U)$ is nontrivial, the global fixed point allows us to turn it into a nontrivial, orientation-preserving action of $H$ on $\R$.  Since $\Homeo^{+}(\R)$ is left orderable, this gives a surjection from $H$ to an infinite left orderable group.   If this is not the case, then $H$ fixes $\del U$ pointwise.  (It is easy to show that a group fixes all prime ends of a region if and only if it fixes the boundary of the region pointwise.)  In this situation we will use Thurston stability.  Let $p$ be a point on $\del U$ and let $D:H \to \GL(2,\R)$ be the derivative at $p$.  That $H$ fixes $\del U$ pointwise means that $p$ is an accumulation point of fixed points, so $H$ fixes both $p$ and a tangent vector at $p$ and the image of $D$ lies in $\Aff^+(\R)$.  Since $\Aff^+(\R)$ is left orderable, the image of $D$ is as well.  We now have the following short exact sequence
$$ 1 \to \ker(D) \to H \to D(H) \to 1$$
with $\ker(D)$ and $D(H)$ left orderable.  Fact \ref{fact} implies that $H$ is left orderable as well.  Thus, in both cases we have a surjection from $H$ to an infinite left orderable group, so by Burns--Hale $G$ is left orderable.
\end{proof}

To see sharpness of the constant here, consider the cyclic group $G$ generated by rotation by $\frac{2\pi}{3}$ about the origin.  For each point $x$ in any ball of radius $r = \frac{1}{\sqrt{3}}k$ centered at the origin we have $\| g(x) - h(x)\| \leq k$ for all $g, h \in G$, but this group is finite and therefore not left orderable.

%
%
%
%

\end{document}